\documentclass[a4paper]{amsart}
\date{August 9, 2019}
\usepackage[dvipdfmx]{graphicx}
\usepackage{latexsym}
\usepackage{amssymb}
\usepackage{amsmath}
\usepackage{amscd}
\usepackage{amsthm}
\usepackage{amsfonts}
\usepackage{mathrsfs}
\usepackage{enumerate}
\usepackage{enumitem}
\usepackage{bm}
\usepackage[usenames]{color}
\usepackage[active]{srcltx}
\title[Behavior of torsion functions]{%
Behavior of torsion functions 
of spacelike curves in Lorentz-Minkowski space}
\author[A.~Honda]{Atsufumi Honda}
\address{%
   Department of Applied Mathematics, 
   Faculty of Engineering, Yokohama National University, 
   79-5 Tokiwadai, Hodogaya, Yokohama 240-8501, Japan
}
\email{honda-atsufumi-kp@ynu.ac.jp}
\thanks{This work was supported by 
JSPS KAKENHI Grant Number 19K14526.}
\usepackage{amsthm}
\theoremstyle{plain}
 \newtheorem{theorem}{Theorem}[section]

 \newtheorem{proposition}[theorem]{Proposition}
 \newtheorem{fact}[theorem]{Fact}
 \newtheorem*{fact*}{Fact}
 \newtheorem{lemma}[theorem]{Lemma}
 \newtheorem{corollary}[theorem]{Corollary}
 \theoremstyle{remark}
 \newtheorem{definition}[theorem]{Definition}
 
 \newtheorem*{acknowledgements}{Acknowledgements}
 \newtheorem{example}[theorem]{Example}
\numberwithin{equation}{section}

 \makeatletter
    
    \@addtoreset{equation}{section}
  \makeatother

\newcommand{\Z}{\mathbb{Z}}
\newcommand{\R}{\mathbb{R}}
\newcommand{\SO}{\operatorname{SO}}
\renewcommand{\O}{\operatorname{O}}
\renewcommand{\L}{\mathbb{L}}
\newcommand{\vect}[1]{\boldsymbol{#1}}

\newcommand{\inner}[2]{\left\langle{#1},{#2}\right\rangle}

\begin{document}
\begin{abstract}
In this paper, 
we introduce the \emph{pseudo-torsion functions}
along spacelike curves whose curvature vector field 
has isolated lightlike points in Lorentz-Minkowski $3$-space,
and prove the fundamental theorem.
Moreover, 
we analyze the behavior of the torsion function at such points.
As a corollary,
we obtain a necessary and sufficient condition
for real analytic spacelike curves to be planar.
\end{abstract}
\maketitle

\section{Introduction}

We denote by $\L^3$ the Lorentz-Minkowski $3$-space
equipped with the Lorentzian inner product $\inner{~}{~}$.
Consider a spacelike regular curve
$\gamma : I \rightarrow \L^3$,
where $I$ is an open interval.
Without loss of generality,
we may assume that
$\gamma$ is parametrized by arclength,
that is, $\inner{\gamma'(s)}{\gamma'(s)}=1$ holds for all $s\in I$,
where the prime means $d/ds$.
Then, we call 
$$
  \vect{\kappa}(s):=\gamma''(s)
$$
the {\it curvature vector field} 
along $\gamma(s)$.
Although the tangent vector $\gamma'(s)$ is spacelike,
$\vect{\kappa}(s)$
may be spacelike, timelike, or lightlike.
We list the known results concerning curves in $\L^3$
depending on the causal characters of $\vect{\kappa}(s)$
as follows:
\begin{itemize}
\item
In the case that $\vect{\kappa}$ is not lightlike on $I$,
the curve $\gamma(s)$ is said to be a {\it spacelike Frenet curve}.
The curvature and torsion functions 
can be defined as in the case of Euclidean $3$-space.
Then, the Frenet-Serret type formula is obtained,
which yields
the fundamental theorem for such spacelike curves
\cite{Walrave, Lopez} 
(cf.\ \S \ref{sec:type-Frenet}).
However, if $\vect{\kappa}$ admits a lightlike point,
such a procedure cannot be proceeded.
\item
In the case that $\vect{\kappa}(s)$ is lightlike 
for each $s\in I$,
although the Frenet frame cannot be defined 
in a similar manner,
one can define another frame
which satisfies the Frenet-Serret type formula.
The fundamental theorem for 
such spacelike curves also holds.
For more details, see
\cite{Walrave, Inoguchi, Lopez}
(cf.\ \S \ref{sec:type-L}).
\item
With respect to timelike or lightlike curves,
for a timelike curve $\gamma : I\to \L^3$,
one can consider the curvature vector field $\vect{\kappa}$ 
in a similar way.
Since $\vect{\kappa}$ is always spacelike,
the curvature and torsion functions 
can be defined in the usual way 
\cite{Walrave, Lopez}.
Also, for a non-degenerate lightlike curve,
a suitable frame can be defined
\cite{Bonnor, Graves, FGL, InoguchiLee, Lopez}
(see also \cite{Vessiot, Walrave, Olszak, Akamine}).
In \cite{FP}, another treatment of the curvatures of curves 
in semi-Euclidean space is given.
Global properties of spacelike curves in $\L^3$
are investigated in \cite{IKT}.
\end{itemize}
Hence,
we may say that
the previous studies have not clarified the structure
of spacelike curves 
whose curvature vector field 
has isolated lightlike points.

In this paper, we deal with 
spacelike curves whose
curvature vector field $\vect{\kappa}(s)$
has isolated lightlike points.
We introduce a torsion-like invariant
called the {\it pseudo-torsion function},
which yields the fundamental theorem 
for such spacelike curves
(Theorem \ref{thm:fundamental}, Corollary \ref{cor:fundamental}).
An application of such the fundamental theorem
for mixed type surfaces in $\L^3$
can be found in \cite{Honda}.

This paper is organized as follows.
In Section \ref{sec:prelim},
we review the several formulas of vectors
and curves in $\L^3$.
In Section \ref{sec:Lk},
we deal with spacelike curves 
having points where curvature vector field 
$\vect{\kappa}(s)$ is lightlike.
At such points,
the torsion function is unbounded, in general.
We obtain the coefficient of the divergent term 
of the torsion function (Theorem \ref{thm:tau-Lk})
for spacelike curves of {\it type $L_k$} 
(cf.\ Definition \ref{def:type-Lk}).
As a corollary,
we obtain a necessary and sufficient condition for 
real analytic spacelike curves with non-zero curvature vector 
to be planar (Corollary \ref{cor:planarity}).

\section{Preliminaries}
\label{sec:prelim}

We denote by $\L^3$ the Lorentz-Minkowski $3$-space
with the standard Lorentz metric $\inner{~}{~}$.
Namely,
$$
  \inner{\vect{x}}{\vect{x}} 
  =  x^2 + y^2 - z^2
$$
holds for each $\vect{x}=(x,y,z)^T\in \L^3$,
where $\vect{x}^T$ stands for the transpose of 
the column vector $\vect{x}$.
A vector $\vect{x}\in \L^3$
is called {\it spacelike\/}
if $\inner{\vect{x}}{\vect{x}}>0$ or $\vect{x} = \vect{0}$.
Similarly, 
if $\inner{\vect{x}}{\vect{x}}<0$
(resp.\ $\inner{\vect{x}}{\vect{x}}=0$),
$\vect{x}$ is called {\it timelike\/} 
(resp.\ {\it lightlike\/}).
For $\vect{x}\in \L^3$, we set
$|\vect{x}|:= \sqrt{|\inner{\vect{x}}{\vect{x}}|}$.

For vectors $\vect{v}, \vect{w} \in \L^3$,
the vector product $\vect{v}\times \vect{w}$
is given by
$\vect{v}\times \vect{w}:=Z \vect{v}\times_{\rm e} \vect{w}$,
where $\times_{\rm e}$ means 
the standard cross product of 
the Euclidean $3$-space $\R^3$,
and we set $Z:={\rm diag}(1,1,-1)$.
Then, it holds that
\begin{align}
\label{eq:scalar-triplet}
\det(\vect{u},\vect{v},\vect{w})&=\inner{\vect{u}}{\vect{v}\times \vect{w}},\\
\label{eq:vector-triplet}
\vect{u}\times (\vect{v}\times \vect{w})
&=\inner{\vect{u}}{\vect{v}}\vect{w}-\inner{\vect{u}}{\vect{w}}\vect{v},\\
\label{eq:area-formula}
\inner{ \vect{v}\times \vect{w} }{ \vect{v}\times \vect{w} }
  &= - \inner{\vect{v}}{\vect{v}} \inner{\vect{w}}{\vect{w}}
     + \inner{\vect{v}}{\vect{w}}^2
\end{align}
for $\vect{u},\vect{v},\vect{w}\in \L^3$.
In particular, 
$\vect{v}\times \vect{w}$ is orthogonal to 
$\vect{v}$ and $\vect{w}$.
To calculate the vector product of a lightlike vector,
the following formula is useful.

\begin{fact}[{\cite[Lemma 4.3]{HST}}]
\label{fact:gaiseki}
Let $\vect{v}\in \L^3$ be a spacelike vector.
Take a lightlike vector $\vect{w}\in \L^3$
such that $\inner{\vect{v}}{\vect{w}}=0$.
Then, either
$\vect{v} \times \vect{w} = |\vect{v}| \, \vect{w}$
or $\vect{v} \times \vect{w} = -|\vect{v}| \, \vect{w}$
holds.
\end{fact}

The isometry group of $\L^3$
is described as the semidirect product
$
  {\rm Isom}(\L^3)=\O(2,1) \ltimes \L^3,
$
where $\O(2,1)$ is the Lorentz group
which consists of square matrices $A$ of order $3$
such that $A^T Z A = Z$.
We set $\SO(2,1)$ and $\SO^+(2,1)$ as
\begin{align*}
  \SO(2,1)&:=\left\{ A \in \O(2,1)\,;\, \det A =1 \right\},\\
  \SO^+(2,1)&:=\left\{ A=(a_{ij}) \in \SO(2,1)\,;\, a_{33} >0 \right\},
\end{align*}
respectively.
Orientation-preserving isometries form the subgroup 
$\SO(2,1)\ltimes \L^3$ 
of the isometry group ${\rm Isom}(\L^3)$.

\subsection{Spacelike curves in $\L^3$}

Let $I$ be an open interval.
A regular curve
$\gamma : I \rightarrow \L^3$
is called {\it spacelike} 
if each tangent vector is spacelike.
By a coordinate change,
we may assume that
$\gamma$ is parametrized by arclength.
That is, 
$\vect{e}(s):=\gamma'(s)$
gives the spacelike tangent vector field of unit length,
where the prime means $d/ds$.
As in the introduction, we call 
$
  \vect{\kappa}(s):=\gamma''(s)
$
the {\it curvature vector field} along $\gamma(s)$.
If $\vect{\kappa}(s)$ is nowhere zero, 
then $\gamma(s)$ is said to be of {\it non-zero curvature vector}.

In a general parametrization $\gamma=\gamma(t)$,
the curvature vector field is written as
$$
  \vect{\kappa}(t)
   =\frac{\dot{\gamma}(t) \times (\dot{\gamma}(t)\times \ddot{\gamma}(t))
   }{|\dot{\gamma}(t)|^4},
$$
where the dot means $d/dt$.
Here, we used \eqref{eq:vector-triplet}.
So we have the following:
\[
\begin{minipage}{0.8\linewidth}
{\it 
Let $\gamma(t)$ be a spacelike curve in $\L^3$,
which may not be of unit speed.
Then, $\gamma(t)$ has non-zero curvature vector
if and only if 
$\dot{\gamma}(t)$ and $\ddot{\gamma}(t)$
are linearly independent.
}\end{minipage}
\]

For a spacelike curve 
with non-zero curvature vector,
we prepare several terminologies as follows:

\begin{definition}\label{def:type-Lk}
Let $\gamma : I \rightarrow \L^3$
be a spacelike curve 
with non-zero curvature vector.
\begin{itemize}
\item
A point $s_0\in I$ is said to be 
a {\it curvature-lightlike point},
if $\vect{\kappa}(s_0)$ is a lightlike vector.
If there exists an open neighborhood 
$J$ of a curvature-lightlike point $s_0$
such that 
$J\setminus\{s_0\}$ consists of 
non-curvature-lightlike points,
then $s_0$ is called 
{\it isolated}.
\item
The curve $\gamma(s)$ is said to be a {\it spacelike Frenet curve},
if it has no curvature-lightlike points.
In particular, 
if $\vect{\kappa}(s)$ is a spacelike 
(resp.\ timelike) vector field along $\gamma(s)$,
then 
$\gamma(s)$ is said to be of {\it type $S$} 
(resp.\ {\it type $T$}\/).
\item
If $\vect{\kappa}(s)$ is a lightlike vector field along $\gamma(s)$,
then
$\gamma(s)$ is called a spacelike curve of {\it type $L$}.
\item
The function $\theta(s)$ defined by
$$
\theta(s):=
\inner{\vect{\kappa}(s)}{\vect{\kappa}(s)}
$$
is called the {\it causal curvature function}.
Then, 
$\gamma(s)$ is of type $S$ (resp.\ type $T$, type $L$\/)
if and only if 
$$  \theta(s)>0\quad
  (\text{resp}.~\theta(s)<0,~\theta(s)\equiv0)
$$
holds on $I$.
\item
Letting $k\in \Z$ be a positive integer,
$\gamma(s)$ is said to be 
of {\it type $L_k$} at $s_0\in I$,
if 
$$\theta(s_0) 
= \cdots= \theta^{(k-1)}(s_0)=0,
\quad
\theta^{(k)}(s_0)\neq0.$$
\end{itemize}
\end{definition}

Let us remark that, 
if $\gamma(s)$ is of type $L_k$ at $s_0\in I$,
then $s_0$ is an isolated curvature-lightlike point.
Moreover, 
if $\gamma:I\to \L^3$ is real analytic,
then the type must be either $S$, $T$, $L$ or $L_k$
(cf.\ Corollary \ref{cor:planarity}).

In the following,
we review the fundamental properties of 
spacelike Frenet curves or 
spacelike curves of type $L$,
as explained in \cite{Lopez}.

\subsection{Spacelike Frenet curves}
\label{sec:type-Frenet}

Let $\gamma : I \to \L^3$ 
be a spacelike Frenet curve
parametrized by arclength. 
Denote by 
$\vect{e}(s)=\gamma'(s)$
$\left({\rm resp}.\ \vect{\kappa}(s)=\gamma''(s) \right)$
the unit tangent vector field 
(resp.\ the curvature vector field) 
along $\gamma(s)$.

The curvature function is defined as 
$
\kappa(s):=\sqrt{|\theta(s)|}
\, (=|\vect{\kappa}(s)|).
$
We set the sign $\sigma_{\gamma}\in \{+1,-1\}$
as
$\sigma_{\gamma}=-1$ (resp.\ $+1$)
if $\gamma$ is of type $S$ (resp.\ type $T$).
Then, 
$$
  \vect{n}(s):=\frac1{\kappa(s)}\vect{\kappa}(s),\quad
  \vect{b}(s):= \sigma_{\gamma} \vect{e}(s) \times \vect{n}(s)
$$
are respectively
called the principal normal vector field,
and
the binormal vector field,
which satisfy $\det (\vect{e}, \vect{n}, \vect{b}) =1$.
The torsion function is defined as
$\tau(s):= \inner{\vect{n}'(s)}{\vect{b}(s)}$.
The system of the equations,
$\vect{e}' = \kappa \vect{n}$,
$\vect{n}' = \sigma_{\gamma}(\kappa \vect{e} + \tau \vect{b})$,
and
$\vect{b}' = \sigma_{\gamma} \tau \vect{n}$,
is called the Frenet-Serret type formula.
Then, it holds that
\begin{equation}\label{eq:curve-det}
  \det (\gamma' ,\, \gamma'',\, \gamma''') = -\theta \tau.
\end{equation}

\subsection{Spacelike curves of type $L$}
\label{sec:type-L}

Let $\gamma : I \to \L^3$ 
be a unit speed spacelike curve of type $L$.
By definition,
$\vect{\kappa}(s)$ is a lightlike vector field.
Then, there exists a lightlike vector field
$\vect{\beta}(s)$ 
such that
$$
  \inner{\vect{e}(s)}{\vect{\beta}(s)} =0,
  \quad
  \inner{\vect{\kappa}(s)}{\vect{\beta}(s)}=1.
$$
Such a vector field $\vect{\beta}(s)$ is uniquely determined.
We call $\vect{\beta}(s)$ 
the {\it pseudo-binormal vector field}.
Then, $\mu(s) := - \inner{\vect{\kappa}'(s)}{\vect{\beta}(s)}$ is called 
the {\it pseudo-torsion function}\footnote{%
In \cite{Lopez}, $-\mu(s)$ is called 
the pseudo-torsion function.
However, to maintain the consistency with 
the notations for spacelike curves of type $L_k$,
we use the minus sign.
We also remark that 
the terminology `pseudo-binormal vector field'
is not used in \cite{Lopez}.}.
The system of the equations,
$\vect{e}' = \vect{\kappa}$,
$\vect{\kappa}' = -\mu \vect{\kappa}$,
and
$\vect{\beta}' = - \vect{e} + \mu \vect{\beta}$,
is called the Frenet-Serret type formula.
In terms of the frame 
$\mathcal{F}:=( \vect{e},\,\vect{\kappa},\,\vect{\beta} )$,
the formula
is rewritten as 
$$
  \mathcal{F}'=\mathcal{F}
  \begin{pmatrix}
  0 & 0 & -1\\
  1 & -\mu & 0\\
  0 & 0 & \mu
  \end{pmatrix}.
$$

\section{Spacelike curves whose curvature vector field
has isolated lightlike points}
\label{sec:Lk}

In this section, we investigate 
spacelike curves having 
isolated curvature-lightlike points.
%
%
%

\subsection{Pseudo-torsion function and Frenet-Serret type formula}
Let $\gamma : I \to \L^3$ be a unit speed spacelike curve
with non-zero curvature vector.
Denote by 
$\vect{e}(s)=\gamma'(s)$
$\left({\rm resp}.\ \vect{\kappa}(s)=\gamma''(s) \right)$
the unit tangent vector field 
(resp.\ the curvature vector field) 
along $\gamma(s)$.
The causal curvature function is given by
$\theta(s)=\inner{\vect{\kappa}(s)}{\vect{\kappa}(s)}$.

Let $s_0\in I$ be a curvature-lightlike point,
namely, $\vect{\kappa}(s_0)$ is a lightlike vector.
By Fact \ref{fact:gaiseki},
$$
\vect{e}(s_0)\times\vect{\kappa}(s_0) = \epsilon \vect{\kappa}(s_0)
$$
holds for some $\epsilon\in \{+1,-1\}$.
We call $\epsilon$ the {\it sign}.
When we emphasize $\gamma(s)$
and $s_0$,
we also denote by $\epsilon={\rm sgn}(\gamma,s_0)$.



\begin{lemma}\label{lem:m-k}
Let $\gamma(s)$ be a spacelike curve with non-zero curvature vector,
and $s_0\in I$ be an isolated curvature-lightlike point.
Then, there exists 
a non-vanishing vector field 
$\vect{\xi}(s)$ along $\gamma(s)$ 
defined on a neighborhood of $s_0\in I$
such that
$$
  \vect{e}(s)\times\vect{\kappa}(s) 
  = \epsilon \vect{\kappa}(s) + \theta(s)\,\vect{\beta}(s)
$$
holds.
{\rm (}We call such a vector field $\vect{\beta}(s)$
the \emph{pseudo-binormal vector field} along $\gamma(s)$.{\rm )}
\end{lemma}


\begin{proof}
Without loss of generality,
we may assume that 
$\gamma(s)$ is parametrized by arclength.
Since $\inner{\vect{e}(s)}{\vect{e}(s)}=1$,
there exist 
$\alpha_1(s)$, $\alpha_2(s)$
such that
\begin{equation}\label{eq:e-sub}
  \vect{e}(s)=\begin{pmatrix}
    \cosh (\alpha_1(s)) \cos (\alpha_2(s))\\
    \cosh (\alpha_1(s)) \sin (\alpha_2(s))\\
    \sinh (\alpha_1(s))
  \end{pmatrix}
\end{equation}
holds.
Then, $\vect{\kappa}(s)=\vect{e}'(s)$
is given by
\begin{equation}\label{eq:n-sub}
  \vect{\kappa}(s)=\begin{pmatrix}
          \alpha_1' \sinh \alpha_1 \cos \alpha_2 
          - \alpha_2'\cosh \alpha_1 \sin \alpha_2\\
          \alpha_1' \sinh \alpha_1 \sin \alpha_2 
          + \alpha_2'\cosh \alpha_1 \cos \alpha_2\\
          \alpha_1' \cosh \alpha_1
      \end{pmatrix}.
\end{equation}
The causal curvature function $\theta(s)$
is calculated as
$$
  \theta(s)
  =   (\alpha_2'\cosh\alpha_1 -\alpha_1')
  (\alpha_2'\cosh\alpha_1 +\alpha_1').
$$
In the case of $\epsilon=1$,
\eqref{eq:e-sub} and \eqref{eq:n-sub} yield
\begin{equation}\label{eq:delta-plus}
  \vect{e}(s)\times\vect{\kappa}(s) 
  = \vect{\kappa}(s) + (\alpha_1' + \alpha_2' \cosh \alpha_1 ) \vect{\xi}(s),
\end{equation}
where 
$$
  \vect{\xi}(s):= -
     \begin{pmatrix}
       \sinh \alpha_1 \cos \alpha_2 - \sin \alpha_2\\
       \sinh \alpha_1 \sin \alpha_2 + \cos \alpha_2\\
       \cosh \alpha_1
     \end{pmatrix}.
$$
Since $\vect{\xi}(s)\neq \vect{0}$,
we have
$
  \alpha_1' +\alpha_2' \cosh \alpha_1=0
$
at $s_0$.
We remark that
$\alpha_1' -\alpha_2' \cosh \alpha_1\ne0$
holds 
by the assumption of non-zero curvature vector 
$\vect{\kappa}(s)\ne\vect{0}$.
Then, \eqref{eq:delta-plus} yields that
$
  \vect{e}(s)\times\vect{\kappa}(s)
  = \vect{\kappa}(s) + \theta(s) \vect{\beta}(s),
$
where we set 
$$
\vect{\beta}(s)
:= \frac{1}{\alpha_1'(s) -\alpha_2'(s) \cosh \alpha_1(s)}
\vect{\xi}(s),
$$
and hence, we obtain the desired result.
A similar argument can be applied in the case of $\epsilon=-1$.
\end{proof}

\begin{proposition}\label{prop:Frenet-Lk}
The pseudo-binormal vector field $\vect{\beta}(s)$
is a lightlike vector field satisfying
\begin{align}
\label{eq:frame-relation1}
  &\inner{\vect{e}(s)}{\vect{\beta}(s)}
  =0,\quad
  \inner{\vect{\kappa}(s)}{\vect{\beta}(s)}=1,\\
\label{eq:frame-relation2}
  &\det(\vect{e},\,\vect{\kappa},\,\vect{\beta})= \epsilon.
\end{align}
Moreover, setting
$
  \mu(s) := -\inner{\vect{\kappa}'(s)}{\vect{\beta}(s)},
$
we have
\begin{equation}\label{eq:Frenet-Lk}
\vect{\kappa}' = -\theta \vect{e}  -\mu \vect{\kappa} 
+ \left(\mu\theta + \frac1{2}\theta' \right) \vect{\beta},
\quad
\vect{\beta}' = -\vect{e} + \mu \vect{\beta}.
\end{equation}
\end{proposition}

\begin{proof}
Since 
$$
  \vect{\beta}(s)
  = - \frac{1}{\theta(s)}\left(
  \epsilon \vect{e}(s)\times\vect{\kappa}(s) - \vect{\kappa}(s)
  \right)
$$
holds for $s\neq s_0$, 
we can verify that $\vect{\beta}(s)$ 
is a lightlike vector for each $s\neq s_0$.
By the continuity, we have 
that $\vect{\beta}(s)$ is a lightlike vector field.
Similarly, \eqref{eq:frame-relation1} holds.
Since
$
\det(\vect{e},\,\vect{\kappa},\,\vect{\beta}) 
= \inner{ \vect{e} \times \vect{\kappa} }{\vect{\beta}}
= \inner{\epsilon(\vect{\kappa}-\theta \vect{\beta}) }{\vect{\beta}}
= \epsilon,
$
we obtain \eqref{eq:frame-relation2}.
With respect to \eqref{eq:Frenet-Lk},
set $\vect{\kappa}' = P \vect{e} + Q \vect{\kappa} + R \vect{\beta}$.
Taking the inner products of 
$\vect{e}$,  $\vect{\kappa}$,  $\vect{\beta}$,
we have
$P= -\theta$,
$Q= -\mu$,
$R= \mu\theta + \theta'/2 $,
where we used \eqref{eq:frame-relation1}.
Similarly, we have
$\vect{\beta}' = -\vect{e} + \mu \vect{\beta}$,
and hence \eqref{eq:Frenet-Lk} holds.
\end{proof}

In the matrix form, 
\begin{equation}\label{eq:Lk-matrix}
  \mathcal{F}'=\mathcal{F}
  \begin{pmatrix}
  0 & -\theta & -1\\
  1 & -\mu & 0\\
  0 & \mu\theta + \frac1{2}\theta' & \mu
  \end{pmatrix}
\end{equation}
holds, where we set 
$\mathcal{F}:=( \vect{e},\,\vect{\kappa},\,\vect{\beta} )$.
We call $\mu(s)$ 
the {\it pseudo-torsion function}.
If we emphasize the sign $\epsilon ={\rm sgn}(\gamma,s_0)$,
we write $\mu(s)=\mu_{\epsilon}(s)$.

\begin{example}[{\cite[Remark 2.2]{Lopez}}]
\label{ex:Lopez-L1}
We set $\gamma : (-\infty,-1) \to \L^3$ as
$$
  \gamma(s)
  :=\begin{pmatrix}
  \cos s + s \sin s\\ 
  \sin s-s\cos s\\
  \left( s\sqrt{s^2-1} - \log\left| s+ \sqrt{s^2-1}\right| \right)/2
  \end{pmatrix}.
$$
The causal curvature function
$\theta(s)$ can be calculated as
$\theta(s)=(s^4-s^2-1)/(s^2-1)$.
Then, we may check that
$\theta(s_0)=0$ and $\theta'(s_0)\neq0$
hold at $s_0=-\sqrt{1+\sqrt{5}}/\sqrt{2}~(\fallingdotseq -1.272)$.
Hence, $\gamma(s)$ is of type $L_1$ at $s_0$.
The pseudo-torsion function $\mu(s)$ can be calculated as
$$
\mu(s) =
 \frac{\sqrt{s^2-1} \left(s^6-2 s^4-2 s^2+2\right)-s \left(s^4-2 s^2+2\right)}
{\left(s^2-1\right) \left(s^4-s^2-1\right)}.
$$
Although the denominator of $\mu(s)$
has zero at $s=s_0$,
we can check that $\mu(s)$ can be analytically extended
across $s=s_0$,
and $\lim_{s\to s_0}\mu(s)=\sqrt{5 \sqrt{5}-11}/(2\sqrt{2})$ holds.
\end{example}

\subsection{Fundamental theorem}

Let $\gamma(s)$ be a spacelike curve 
with non-zero curvature vector,
$s_0\in I$ 
be a curvature-lightlike point,
and $\epsilon={\rm sgn}(\gamma,s_0)\in \{+1,-1\}$
be the sign.
Take an isometry $T$ of $\L^3$,
and set 
$$\bar{\gamma}(s) := T \circ \gamma (s).$$
Then, the causal curvature function $\bar{\theta}(s)$
of $\bar{\gamma}(s)$ coincides with 
that of $\gamma(s)$.
Hence, $s_0\in I$ 
is also a curvature-lightlike point
for $\bar{\gamma}(s)$.

\begin{lemma}\label{lem:rigid-motion}
Let $\bar{\epsilon}={\rm sgn}(\bar{\gamma},s_0)\in \{+1,-1\}$
be the sign of $\bar{\gamma}(s)$ at $s_0$.
If $T$ is orientation preserving
$($resp.\ orientation reversing$)$,
we have 
$\bar{\epsilon}=\epsilon$.
$($resp.\ $\bar{\epsilon}=-\epsilon)$.
Moreover, the pseudo-torsion 
$\bar{\mu}(s)$ 
of $\bar{\gamma}(s)$
coincides with 
that $\mu(s)$ of $\gamma(s)$.
\end{lemma}

We omit the proof
since this lemma can be verified directly
by using the following formula:
$$
  (T\vect{v})\times (T\vect{w})
  = (\det T)\,T(\vect{v}\times \vect{w}) 
$$
holds for 
$\vect{v},\vect{w} \in \L^3$, 
$T\in \O(2,1)$.

By a parallel translation,
we may assume that $\gamma(s_0)=\vect{0}$.
We would like to find 
an orientation preserving isometry $T\in \SO(2,1)$
of $\L^3$ so that the frame 
$$\mathcal{F}:=(\vect{e},\,\vect{\kappa},\,\vect{\beta})$$
has simplified form at $s_0\in I$.
First, we can find $T_1\in \SO^+(2,1)$
such that 
$T_1\vect{e}(s_0)=(1,0,0)^T$.
Since 
$\vect{\kappa}(s_0),\,\vect{\beta}(s_0)$
are lightlike and perpendicular to 
$\vect{e}(s_0)$,
we have 
\begin{align*}
T_1\vect{\kappa}(s_0)=a(0,1,\pm 1)^T,\qquad
T_1\vect{\beta}(s_0)=\frac1{2a}(0,1,\mp 1)^T,
\end{align*}
where $a\in \R$ is a non-zero constant.
Hence, there exists $T_2\in \SO^+(2,1)$
such that 
$T_2T_1\mathcal{F}$ 
at $s_0\in I$
is given by either $E_+$, $E_-$, $E_+'$ or $E_-'$,
where
\begin{align*}
  E_{\sigma_1}
  :=\begin{pmatrix}
  1&0&0\\
  0&\sigma_1/\sqrt{2}&\sigma_1/\sqrt{2}\\
  0&- 1/\sqrt{2}& 1/\sqrt{2}
  \end{pmatrix},\qquad
  E_{\sigma_2}'
  :=\begin{pmatrix}
  1&0&0\\
  0&-\sigma_2/\sqrt{2}&-\sigma_2/\sqrt{2}\\
  0&1/\sqrt{2}&-1/\sqrt{2}
  \end{pmatrix},
\end{align*}
and $\sigma_1, \sigma_2 \in \{+1,-1\}$.
If we set $S\in \SO(2,1)$ as
$S:={\rm diag}(1,-1,-1)$,
then $E_\pm=SE_\pm'$ hold.
Hence, setting $T:=ST_2T_1\in \SO(2,1)$,
we have the following:
\[
\begin{minipage}{0.8\linewidth}
{\it 
Let $\gamma(s)$ be a spacelike curve 
with non-zero curvature vector,
$s_0\in I$ 
be a curvature-lightlike point,
and $\epsilon={\rm sgn}(\gamma,s_0)$
be the sign.
Then there exists $T\in \SO(2,1)$ 
such that 
$T\mathcal{F}(s_0)=E_\epsilon$.
}\end{minipage}
\]

By the existence and uniqueness 
of the solution to the ODE
\eqref{eq:Lk-matrix}
with the initial condition
$\mathcal{F}(s_0)=E_+$ or $\mathcal{F}(s_0)=E_-$,
we obtain
the fundamental theorem of 
spacelike curves 
having isolated curvature-lightlike points:

\begin{theorem}
\label{thm:fundamental}
Let $I$ be an open interval,
and fix $s_0\in I$.
Also let $\theta(s)$, $\mu(s)$ 
be two smooth functions defined on $I$ 
so that 
the zero of $\theta(s)$ is $s_0$ only.
For each choice of the sign $\epsilon\in \{+1,-1\}$,
there exists a unit speed spacelike curve 
$\gamma(s) : I \to \L^3$
with non-zero curvature vector
such that 
\begin{itemize}
\item[(a)]
the causal curvature and 
pseudo-torsion functions
coincide with
$\theta(s)$ and $\mu(s)$, respectively,
\item[(b)]
the sign ${\rm sgn}(\gamma,s_0)$
coincides with $\epsilon$.
\end{itemize}
Moreover, such a curve is unique up to 
orientation preserving isometries of $\L^3$.
\end{theorem}

Since the steps to follow are similar as 
in Euclidean space
(e.g.\ see \cite[Theorem 5.2]{UY}), we omit the proof.

\begin{corollary}[Fundamental theorem of 
spacelike curves with isolated curvature-lightlike points]
\label{cor:fundamental}
Let $I$ be an open interval,
and fix $s_0\in I$.
Also let $\theta(s)$, $\mu(s)$ 
be two smooth functions defined on $I$ 
so that 
the zero of $\theta(s)$ is $s_0$ only.
Then there exists a unit speed spacelike curve 
$\gamma(s) : I \to \L^3$
with non-zero curvature vector
such that 
the causal curvature and 
pseudo-torsion functions
coincide with
$\theta(s)$ and $\mu(s)$, respectively.
Moreover, such a curve is unique up to 
isometries of $\L^3$.
\end{corollary}

\begin{proof}
In Theorem \ref{thm:fundamental},
we proved that, for each choice of 
$\epsilon\in\{+1,-1\}$,
there exists a unit speed spacelike curve 
$\gamma^{\epsilon}(s) : I \to \L^3$
satisfying the conditions (a) and (b).
Moreover, such a curve $\gamma^{\epsilon}(s)$
is unique up to orientation preserving isometries of $\L^3$.
As we see in Lemma \ref{lem:rigid-motion},
$\gamma^{+}(s)$ and $\gamma^{-}(s)$ 
are transferred to each other by isometries of $\L^3$.
Hence, we have the uniqueness.
\end{proof}

\subsection{Behavior of torsion function}

Let $\gamma(s) : I \to \L^3$
be a spacelike curve 
having an isolated curvature-lightlike point $s_0\in I$.
Since $\theta(s)\ne0$ for $s$ near $s_0$,
the spacelike curve
$\gamma(s)$ $(s\ne s_0)$ is a Frenet curve.
At $s\ne s_0$, the torsion function is defined,
and calculated as follows:

\begin{lemma}
\label{lem:tau-Lk}
Let $\gamma(s) : I \to \L^3$
be a unit speed spacelike curve
with non-zero curvature vector,
and $s_0\in I$ 
be an isolated curvature-lightlike point.
Then the torsion function $\tau(s)$
of $\gamma(s)$ $(s\ne s_0)$ is given by
$$
  \tau(s) = -\mu(s) - \frac1{2}\frac{\theta'(s)}{\theta(s)}.
$$
\end{lemma}

\begin{proof}
Substituting 
$\gamma'=\vect{e}$, 
$\gamma''=\vect{\kappa}$, 
$\gamma'''=\vect{\kappa}'$
into 
$\det (\gamma',\gamma'',\gamma''') = -\theta \tau$
as in \eqref{eq:curve-det},
and applying \eqref{eq:Frenet-Lk},
we may verify the desired identity.
\end{proof}

\begin{theorem}
\label{thm:tau-Lk}
For a positive integer $k$,
let $\gamma(s) : I \to \L^3$
be a spacelike curve
of type $L_k$ at $s_0\in I$
parametrized by arclength.
Then 
$$
   \lim_{s\to s_0} (s-s_0) \, \tau (s) = -\frac{1}{2}k
$$
holds.
In particular,
$\tau(s)$ is unbounded at $s_0$.
\end{theorem}

\begin{proof}
By the definition of type $L_k$ in Definition \ref{def:type-Lk},
the division lemma\footnote{The division lemma
is also called {\it Hadamard's lemma},
cf.\ \cite[Lemma 3.4]{BG}.} 
yields that
there exists a function $\hat{\theta}(s)$
such that 
$$
  \theta(s)=(s-s_0)^{k}\,\hat{\theta}(s)\qquad
  (\hat{\theta}(s_0)\neq0)
$$
holds.
%
Together with Lemma \ref{lem:tau-Lk},
we have
$$
  \tau = -\mu - \frac1{2}\left( \frac{k}{s-s_0} 
  + \frac{\hat{\theta}'}{\hat{\theta}} \right).
$$
Hence,
$(s-s_0)\,\tau(s) \to - k/2$
holds as $s\to s_0$.
\end{proof}

We remark that, 
for a planar curve having a cusp singularity,
the behavior of curvature function at the cusp
is investigated \cite{SU}.

\begin{example}
Let $\gamma(s)$ be the spacelike curve
of type $L_1$ given in 
Example \ref{ex:Lopez-L1}.
The torsion function $\tau(s)$
can be calculated as
$$
  \tau(s)=\frac{-s^6+2 s^4+2 s^2-2}{\sqrt{s^2-1} \left(s^4-s^2-1\right)},
$$
which is unbounded at $s_0=-\sqrt{1+\sqrt{5}}/\sqrt{2}$.
We can check that 
$
\lim_{s\to s_0}(s-s_0)\tau(s)=-1/2,
$
which verifies Theorem \ref{thm:tau-Lk}.
\end{example}


As a corollary of Theorem \ref{thm:tau-Lk},
we obtain the planarity condition
for analytic spacelike curves (Corollary \ref{cor:planarity}).
For spacelike curves of type $S$, $T$ or $L$,
the planarity condition is given as follows:

\begin{fact}[{\cite[Theorem 2.3]{Lopez}, 
\cite[Corollary 3.2]{IPS}, \cite[Remark 7]{daSilva}}]
\label{fact:Lopez-planar}
Let $\gamma : I \to \L^3$ 
be a spacelike curve.
\begin{itemize}
\item 
Assume that $\gamma(s)$ is a Frenet curve.
Then, $\gamma(s)$ is included in an affine plane 
if and only if 
its torsion function is identically zero.
\item
Assume that $\gamma(s)$ is of type $L$.
Then, $\gamma(s)$ is a planar curve
included in a lightlike plane.
\end{itemize}
\end{fact}

So, it is natural to ask 
the planarity condition
for general spacelike curves.
In analytic case, we have the following:

\begin{corollary}\label{cor:planarity}
A real analytic spacelike planar curve 
with non-zero curvature vector 
must be either
a spacelike Frenet curve with vanishing torsion or 
a spacelike curve of type $L$.
\end{corollary}

\begin{proof}
Let $\gamma(s) : I \to \L^3$ 
be a real analytic spacelike curve with non-zero curvature vector.
By the analyticity,
$\gamma(s)$ must be of type $S$, $T$, $L$ or $L_k$.
By Fact \ref{fact:Lopez-planar},
it suffices to show that 
if $\gamma(s)$ is of type $L_k$,
then $\gamma(s)$ never be planar.
Assume that $\gamma(s)$ is of type $L_k$ at $s_0\in I$,
and that $\gamma(s)$ is planar.
Since $\gamma(s)$ $(s\ne s_0)$ is a Frenet curve
which is included in an affine plane,
the torsion $\tau(s)$ of $\gamma(s)$ $(s\ne s_0)$ is
identically zero.
However, by Theorem \ref{thm:tau-Lk},
$\tau(s)$ is unbounded at $s=s_0$,
which is a contradiction.
Hence, any spacelike curve of type $L_k$ never be planar.
\end{proof}

\begin{acknowledgements}
The author expresses his gratitude to 
Junichi Inoguchi for helpful comments.
He also would like to thank 
Shyuichi Izumiya and Luiz C.\ B.\ da Silva
for giving him several informations on
Fact \ref{fact:Lopez-planar}.
\end{acknowledgements}


\end{document}